\documentclass{article}[11pt]
\usepackage{amsmath,authblk,graphicx}
\usepackage{amssymb}
\usepackage{amsthm}
\usepackage{verbatim}
\usepackage{appendix,latexsym,amsfonts,amsmath,amssymb,graphicx,hyperref,amsthm,authblk,comment,epstopdf}
\usepackage{graphicx}
\usepackage{pinlabel}
\usepackage{amssymb,amsmath,comment}
\usepackage{amsthm}
\usepackage{graphicx}
\usepackage{pinlabel}
\theoremstyle{plain}
\newtheorem{Theorem}{Theorem}[section]
\newtheorem{lemma}[Theorem]{Lemma}
\newtheorem{corollary}[Theorem]{Corollary}
\newtheorem{proposition}[Theorem]{Proposition}
\theoremstyle{definition}
\newtheorem{definition}[Theorem]{Definition}
\newtheorem{remark}[Theorem]{Remark}
\numberwithin{equation}{section}

\def\eps{\varepsilon}

\def\til{\widetilde}
\def\ha{\widehat}
\def\sem{\setminus}
\def\lin{\overline}

\newcommand{\F}{{\cal F}}
\def\pa{\partial}

\numberwithin{equation}{section}

  \DeclareMathOperator{\Imm}{Im}
\DeclareMathOperator{\mA}{m}

\def \H{\mathbb{H}}
\def \E{\mathbb{E}}
\def \P{\mathbb{P}}
\def \C{\mathbb{C}}

\def \N{\mathbb{N}}
\def\I {{\bf 1}}
\def \R{\mathbb{R}}

\def \({\left(}
\def \){\right)}
\def \[{\left[}
\def \]{\right]}

\newcommand{\BGE}{\begin{equation}}
\newcommand{\BGEN}{\begin{equation*}}
\newcommand{\EDE}{\end{equation}}
\newcommand{\EDEN}{\end{equation*}}

\begin{document}
\title{Decomposition of backward SLE in the capacity parameterization}
\author{Benjamin Mackey\thanks{Michigan State University and Cleveland State University.} and Dapeng Zhan\thanks{Michigan State University. Partially supported by NSF grant  DMS-1056840 and Simons Foundation grant \#396973.}}
\maketitle

\begin{abstract}
  We prove that, for $\kappa\le 4$, backward chordal SLE$_\kappa$ admits backward chordal SLE$_\kappa(-4,-4)$ decomposition for the capacity parametrization.
  This means that, for any bounded measurable subset $U\subset Q_4:=\R_+\times\R_-$, if we integrate the laws of extended backward chordal SLE$_\kappa(-4,-4)$ with different pairs of force points $(x,y)$ against some suitable density function $G(x,y)$ restricted to $U$, then we get a measure, which is absolutely continuous with respect to the law of backward chordal SLE$_\kappa$, and the Radon-Nikodym derivative is a constant depending on $\kappa$ times the capacity time that the generated welding curve $t\mapsto (d_t,c_t)$ spends in $U$, where $d_t>0>c_t$ are the pair of points that are swallowed by the process at time $t$. For the forward SLE curve, a similar analysis has been done for SLE in the natural parametrization (\cite{Field} $\kappa \leq 4$, \cite{Zhan decomp} $\kappa <8$), and for the capacity parametrization (\cite{Zhan decomp} $\kappa < \infty$).
 \end{abstract}

\section{Introduction}
	Chordal Schramm-Loewner evolution, or SLE, is a path in the upper half plane generated by plugging a time scaled Brownian motion into the chordal Loewner differential equation. In some cases, the reverse, or backward, flow of the Loewner equation is easier to study and can be used to answer questions about the regular, or forward, SLE process. Analysis of the backward flow was used to show existence of the SLE$_{\kappa}$ curve for $\kappa \neq 8$ (\cite{Rhode Schramm}). In \cite{Law multifrac}, a multifractal analysis  is used to study moments for the backward SLE flow, which is used to provide a new proof of the Hausdorff dimension of an SLE path.

While the backward SLE process, abbreviated as BSLE$_{\kappa}$, generates a family of random paths, they do not form a good global object with which to study the whole process. Instead, for $\kappa \leq 4$, the BSLE$_{\kappa}$ process generates a random continuous homeomorphism $\phi: \R \to \R$ called the conformal welding defined by the following: for each $t>0$, there are exactly two points $c_t<0<d_t$ that are swallowed by the process at time $t$, then $\phi(c_t)=d_t$ and $\phi(d_t)=c_t$. BSLE and the conformal welding have been coupled with the Gaussian free field and what is called the Liouville quantum zipper \cite{Sheffield zipper}, where the welding of the real line onto the backward SLE$_{\kappa}$ traces is seen as the conformal image of gluing random surfaces together. In \cite{RZ}, the welding was proven to be time reversible, which is analogous to the fundamental reversibility property of the forward SLE$_{\kappa}$ path \cite{Zhan Reversal} \cite{Imaginary geometry 3}. This result is then used to study ergodic properties of the tip of a forward SLE$_{\kappa}$ in \cite{ZhanE}.


The goal of this paper is to establish a result for BSLE$_{\kappa}$, which is similar to the results in \cite{Zhan decomp} for the forward SLE$_{\kappa}$. In \cite{Zhan decomp}, a family of Green's functions $G^{\alpha}(z)$ are constructed, each of which
is associated with an SLE$_\kappa(\rho)$ process with an interior force point $z$ via Girsanov Theorem. The SLE$_\kappa(\rho)$ curve ends at the force point $z$, and an extended SLE$_\kappa(\rho)$ is defined by continuing this SLE$_\kappa(\rho)$ curve with a chordal SLE$_\kappa$ curve in the remaining domain from $z$ to $\infty$. The law of the extended SLE$_\kappa(\rho)$ is denoted by $P_z^\rho$. Given a bounded measurable set $U$ in the upper half plane, a new measure $\P_U^{\rho}$ is constructed by integrating $\P_z^{\rho}$ against $\I_U(z)G^{\rho}(z)A(dz)$, where $A$ is the Lebesgue measure.
For one particular parameter $\rho=\kappa-8$, this measure is absolutely continuous with respect the law of the chordal SLE$_\kappa$, and the Radon-Nikodym derivative is the Minkowski content of the path in $U$. This holds for all $\kappa \in (0,8)$, and extends an earlier result in \cite{Field}. For a different parameter $\rho=-8$, the integrated measure is also absolutely continuous and the Radon-Nikodym derivative is a constant times the capacity time the path spends in $U$. A similar result is also proven for the intersection of the SLE$_{\kappa}$ path with the boundary for $\kappa \in (4,8)$. In all of these cases, the results follow from more precise theorems about path decomposition. This work has recently been used to construct SLE loop measures \cite{Zhan SLE loop}.

In this paper, we consider the relation between BSLE$_\kappa$ and extended BSLE$_\kappa(\rho_+,\rho_-)$ with force points $x>0>y$, whose law is denoted by $\P^{\rho_+,\rho_-}_{x,y}$. We expect that for certain parameters $(\rho_+,\rho_-)$, for any bounded measurable set $U$ in the fourth quadrant, the measure $\P_U^{\rho_+,\rho_-}$ defined as $\int_U G^{(\rho_+,\rho_-)}(x,y) \P^{\rho_+,\rho_-}_{x,y}dxdy$ is absolutely continuous with respect to the law of BSLE$_\kappa$, where $G^{(\rho_+,\rho_-)}$ is the Green's function associated with BSLE$_\kappa(\rho_+,\rho_-)$ via Girsanov Theorem. This is done for one parameter: $\rho_+=\rho_-=-4$ in the case $\kappa\le 4$. Moreover, we prove that the Radon-Nikodym derivative is a constant times the capacity time that the welding curve $t\mapsto (d_t,c_t)$ spends in $U$.

\section{Preliminary}
\subsection{Backward SLE} \label{section: BSLE}
We now review backward Loewner equations. There are two versions of such equations: chordal and radial. This work will focus on backward chordal Loewner equation, and so will omit the word ``chordal''.



For a real valued continuous function $\lambda_t$, $0\le t<T$, where $T\in(0,\infty]$, the backward Loewner equation driven by $\lambda$ is
\begin{equation}\label{eq: backward leq}
\partial_t f_t(z)=\frac{-2}{f_t(z)-\lambda_t}, \quad f_0(z)=z.
\end{equation}
This differs from the (forward) Loewner equation in \cite{Law} by a minus sign. The process $(f_t)_{t <T}$ is called the backward Loewner process driven by $\lambda$.
For each $z \in \C$, let $\tau_z$ denote the lifetime of the equation started at $z$, and let $S_t=\{z\in\C_:\tau_z\le t\}$, $0\le t<T$. Then $S_t=[c_t,d_t]$, $0\le t<T$, is a continuously increasing family of compact real intervals with $S_0=[c_0,d_0]=\{\lambda_0\}$, and each $f_t$ is a conformal map defined on $\C\sem S_t$ that satisfies $\lim_{z \to\infty} f_t(z)=\infty$ and $f_t(\lin z)=\lin{f_t(z)}$. Let $\H$ denote the upper half-plane $\{z\in\C: \Imm z>0\}$. Then $f_t(\H)\subset\H$, and $L_t:=\H\sem f_t(\H)$ is an $\H$-hull with half-plane capacity $2t$ (\cite{Law}). So we say that the backward Loewner process  given by the definition is parametrized by capacity.

\begin{definition} \label{welding}
We call  $\Phi(t)=\Phi^\lambda(t):=(d_t,c_t)\in [\lambda_0,\infty)\times(-\infty,\lambda_0]$, $0\le t<T$, the conformal  welding curve generated by $\lambda$. In the case when the maps $t\mapsto c_t$ and $t\mapsto d_t$ are both strictly monotonic,  we get an auto-homeomorphism $\phi=\phi^\lambda$ of the interval $S_T:=\bigcup_{0\le t<T} S_t$ such that $\phi(c_t)=d_t$ and $\phi(d_t)=c_t$ for each $t$. Such $\phi$ is called the conformal welding generated by $\lambda$.
\end{definition}

\begin{remark}
	The welding curve $\Phi^\lambda$ exists for all driving function $\lambda$, but the welding $\phi^\lambda$ may not exist. When $\phi^\lambda$ exists,  $\Phi^\lambda$ determines   $\phi^\lambda$, and   $\phi^\lambda$ determines $\Phi^\lambda$ up to parametrization.
\end{remark}


Let $\kappa >0$ and $B_t$ be a standard Brownian motion. The process given by solving (\ref{eq: backward leq}) with $\lambda_t=\sqrt\kappa  B_t$ is called the backward SLE$_{\kappa}$ process, and will be denoted by BSLE$_{\kappa}$.
Suppose $\kappa\in(0,4]$. From the relation between BSLE and SLE (\cite{Rhode Schramm}), it is known that the BSLE$_{\kappa}$ hulls are all simple curves. Thus, BSLE$_\kappa$ generates a random conformal welding $\phi$ of $S_\infty=\R$.
This welding is the main object of study in \cite{RZ}, where it is shown that the random map $x\mapsto 1/\phi(1/x)$ has the same law as $\phi$. This is analogous to the reversibility property of the forward SLE curve (\cite{Zhan Reversal, Imaginary geometry 3}).

We will also consider a variant of backward SLE, called the backward SLE process with force points which was introduced in \cite{Sheffield zipper,RZ}. Let $\rho_+,\rho_-\in\R$.
Let $x>0>y$. The chordal Loewner equation with driving function $\lambda$ which satisfies the SDE:
\BGE d\lambda_t = \sqrt{\kappa}dB_t + \frac{-\rho_+ dt}{\lambda_t-f_t(x)} + \frac{-\rho_- dt}{\lambda_t-f_t(y)}, \quad \lambda_0=0,\label{BSLE-force}\EDE
is called a BSLE$_{\kappa}(\rho_+,\rho_-)$ process started from $(0;x, y )$, where $\rho_+,\rho_-\in \R$ are the force values at the force points $x,y \in \R $. This definition can be extended to include more than two force points, which can also be placed in the interior of $\H$, but this is the amount of generality which will be needed in this paper.


We now describe the Radon-Nikodym derivative of the law of a BSLE$_{\kappa}(\rho_+,\rho_-)$ driving process against the laws of a BSLE$_\kappa$ driving process, which is $(\sqrt{\kappa }B_t)$. Let  $X_t=f_t(x)-\lambda_t>0$ and $Y_t=f_t(y)-\lambda_t<0$.

\begin{proposition}\label{prop: general Greens functions}
Suppose $\lambda_t=\sqrt\kappa B_t$, $0\le t<\infty$.
Fix   $\rho_+,\rho_- \in \R$. Then
$$M^{\rho_+,\rho_-}_t(x,y):=|X_t|^{\frac{\rho_+}{-\kappa}} |Y_t|^{\frac{\rho_-}{-\kappa}} |X_t-Y_t|^{\frac{\rho_+\rho_-}{-2\kappa}} |f_t'(x)|^{\frac{\rho_+(\rho_++4-(-\kappa))}{-4\kappa}} |f_t'(-y)|^{\frac{\rho_-(\rho_-+4-(-\kappa))}{-4\kappa}},\quad 0\le t<\tau_x\wedge \tau_y,$$ is a local martingale such that the process $(\sqrt\kappa B_t)$ weighted by $M^{\rho_+,\rho_-}_t(x,y)/M^{\rho_+,\rho_-}_0(x,y)$ using Girsanov theorem is the driving process of a backward SLE$_{\kappa}(\rho_+,\rho_-)$ process started from $(0; x,y)$.
\end{proposition}

\begin{remark} The proposition is similar to \cite[Theorem 6]{SW}. So we omit its proof. One way to check that the formula here is correct is to note that it differs from the formula of \cite[Theorem 6]{SW} in the case that $n=2$, $z_1=x$, $z_2=y$, $\rho_1=\rho_+$ and $\rho_2=\rho_-$ in two places: 1) the forward map $g_t$ there are replaced by the backward map $f_t$ here; 2) $\kappa$ is replaced by $-\kappa$. These changes make sense since it has been explained in \cite{RZ} that BSLE$_\kappa$ may be viewed as SLE with negative parameter ${-\kappa}$.
\end{remark}

Proposition \ref{prop: general Greens functions} motivates us to define the $(\rho_+,\rho_-)$-BSLE$_{\kappa}$ Green's function: $$G^{\rho_+,\rho_-}(x,y):=M^{\rho_+,\rho_-}_0(x,y)= |x|^{\frac{\rho_+}{-\kappa}}|y|^{\frac{\rho_-}{-\kappa}}|x-y|^{\frac{\rho_+\rho_-}{-2\kappa}},$$ so that the $M^{\rho_+,\rho_-}_t(x,y)$ in Proposition \ref{prop: general Greens functions} can be expressed by $G^{\rho_+,\rho_-}(X_t,Y_t)f_t'(x)^{q_+}f_t'(y)^{q_-}$ for some scaling exponents $q_\pm\in\R$ depending on $\kappa,\rho_\pm$.

Suppose now that $\lambda$ is the driving function of  a BSLE$_{\kappa}(\rho_+,\rho_-)$ process started from $(0;x,y)$ with lifetime $[0,T)$.
From (\ref{eq: backward leq}) we know that
\BGE \pa_t(X_t-Y_t)^2= \pa_t(f_t(x)-f_t(y))^2=\frac{4(X_t-Y_t)^2}{X_tY_t}=-\frac{4(|X_t|+|Y_t|)^2}{|X_t||Y_t|}\le -16, \quad 0\le t<T.\label{X-Y}\EDE
So $T\le (x-y)^2/16$. Since $f_t(x)>\lambda_t>f_t(y)$, $f_t(x)$ is decreasing, and $f_t(y)$ is increasing on $[0,T)$, we see that both $\lim_{t \to T} f_t(x)$ and $\lim_{t \to T} f_t(y)$ converge. As $t\to T$, either $f_t(x)-\lambda_t\to 0$ or $f_t(y)-\lambda_t\to 0$, for otherwise the process can be extended beyond $T$. In either case, we have the convergence of $\lim_{t \to T}\lambda_t$.

\begin{definition} \label{extended} If $\lambda$ is the driving function of  a BSLE$_{\kappa}(\rho_+,\rho_-)$ process started from $(0;x,y)$ with lifetime $[0,T)$, we may extend $\lambda$ continuously to $\R$ such that $\lambda(T+t)-\lambda(T)=\sqrt{\kappa} \ha B_t$, $t\ge 0$, for some standard Brownian motion $\ha B_t$ that is independent of $\lambda(t)$, $0\le t<T$. The backward Loewner process driven by such extended $\lambda$ is called an extended BSLE$_{\kappa}(\rho_+,\rho_-)$ process started from $(0;x,y)$.
\end{definition}

Define $u(t)=-\frac{\kappa}{2} \ln( \frac{X_t-Y_t}{x-y} )$ on $[0,T)$.  From (\ref{X-Y}) we know $u(0)=0$ and $u'(t)=\frac{-\kappa}{X_tY_t}$.  Let
$$W_t=\frac{X_t+Y_t}{X_t-Y_t}\in(-1,1),\quad 0\le t<T.$$
 If $\lim_{t\to T} X_t=0>\lim_{t\to T} Y_t$, then $\lim_{t\to T} W_t=-1$; if $\lim_{t\to T} X_t>0=\lim_{t\to T} Y_t$, then $\lim_{t\to T} W_t=1$.
 Let $\ha W_t=W_{u^{-1}(t)}$ be the time-change of $W$ via $u$. A straightforward It\^o's calculation (cf.\ \cite{RY}) using (\ref{eq: backward leq},\ref{BSLE-force}) shows that $\ha W$ satisfies
\begin{equation}\label{eq: radial Bessel}
d\ha{W}_t =-\sqrt{1-\ha{W}_t^2} d\ha{B}_t -\frac{\rho_++2}{-\kappa}(\ha W_t+1) dt -\frac{\rho_-+2}{-\kappa}(\ha W_t-1) dt,
\end{equation}
where $\ha{B_s}$ is a standard Brownian motion. Equation (\ref{eq: radial Bessel}) is associated with the cosine of a radial Bessel process, and is studied in \cite[Appendix B, Remark 3]{ZhanE}. The process $\ha W_t$ behaves like a squared Bessel process of dimension $\delta_+=\frac{4(\rho_++2)}{-\kappa}$ near $1$, and a squared Bessel process of dimension $\delta_-=\frac{4(\rho_-+2)}{-\kappa}$ near $-1$. Thus, if $\rho_\pm\le -\frac\kappa 2-2$, then $\delta_\pm\ge 2$, and $\ha W_t$ does not tend to either $1$ or $-1$ as $t\to u(T)$, which implies that $W_t$ does not tend to $1$ or $-1$ as $t\to T$. So we obtain the following lemma.

\begin{proposition}\label{lem: welding as}
Suppose $\rho_+,\rho_- \leq -\frac\kappa 2-2$. Then for a BSLE$_{\kappa}(\rho_+,\rho_-)$ process started from $(0;x,y)$ with lifetime $[0,T)$, we have a.s.\ $\lim_{t\to T}X_t=0=\lim_{t\to T}Y_t$ and $u(T):=\lim_{t\to T} u(t)=-\frac{\kappa}{2} \ln\( \frac{X_T-Y_T}{x-y} \)=\infty$. 
\end{proposition}

\begin{remark} \label{<-kappa/2-2}
Suppose $\rho_+,\rho_- \leq -\frac\kappa 2-2$. From the proposition we see that, if $\lambda$ is the driving function for an extended BSLE$_\kappa(\rho_+,\rho_-)$ process started from $(0;x,y)$, and if $T$ is the lifetime of the unextended portion of the process, then $x$ and $y$ are both swallowed at the time $T$ by the extended process. So we get $c_T=y$ and $d_T=x$, i.e., $\Phi^{\lambda}(T)=(x,y)$. Moreover, since $u(T)=\infty$,
we may recover $T$ from the diffusion process $(\ha W_t)$ in (\ref{eq: radial Bessel}) by the formula
\begin{equation}\label{eq: tau eq 2}
T=\frac{(x-y)^2}{4 \kappa}\int_0^{\infty} e^{-\frac{4}{\kappa}t}(1-\ha{W}_t^2)dt.
\end{equation}
\end{remark}


\subsection{Processes with a random lifetime}\label{subsection: random lifetime}
We now review the framework introduced in $\cite{Zhan decomp}$ to study stochastic processes with a random lifetime. We will need to introduce the notation which will be used in this paper, and several propositions will be stated without proof. For complete details, the reader can refer to \cite{Zhan decomp}.

Define the space
$$\Sigma = \bigcup_{0 < T \leq \infty} C[0,T),$$ where $C[0,T)$ is the set of real valued functions which are continuous on $[0,T)$.
For each $f\in\Sigma$, let $T_f$ denote the lifetime of $f$, i.e., $[0,T_f)$ is the definition domain of $f$. For $0\le t<\infty$, let $\Sigma_t=\{f\in\Sigma:T_f>t\}$. We define a filtration $(\F_t)$ on $\Sigma$ such that
$${\F}_t=\sigma\(\{ f \in \Sigma_s, f(s) \in U \}, s \leq t, U \subset \R \text{ is measurable}\).$$
The global $\sigma$-field on $\Sigma$ is $\F=\sigma(\F_t:0\le t<\infty)$. If $\mu,\nu$ are measures on $(\Sigma,{\F})$, then we say that $\nu$ is locally absolutely continuous with respect to $\mu$ if for every $t>0$, $\nu|_{\cal{F}_t \cap \Sigma_t}$ is absolutely continuous with respect to $\mu|_{\cal{F}_t \cap \Sigma_t}$. We will use the notation $\nu \ll \mu$ to mean (global) absolute continuity and $\nu \lhd \mu$ to mean local absolute continuity.

We will  define a few operations on this space, which all measurable on $(\Sigma,\cal{F})$. 
\begin{itemize}
\item \textit{Killing}: For $0<\tau \leq \infty$, define $\cal{K}_{\tau}: \Sigma \to \Sigma$ by $\cal{K}_{\tau}(f)=f|_{[0,\tau_f)}$, where $\tau_f=\min\{\tau, T_f\}$.
\item \textit{Continuation}: Define subspaces of $\Sigma$ by $\Sigma^{\oplus}=\{ f \in \Sigma : T_f<\infty, f(T^-):=\lim_{t \to T_f^-}f(t) \in \R \}$ and $\Sigma_{\oplus}=\{ f \in \Sigma: f(0)=0 \}$. Then we define $\oplus:\Sigma^{\oplus}\times \Sigma_{\oplus} \to \Sigma$ by
$$ f \oplus g (t)= \begin{cases}
f(t), & 0 \leq t < T_f\\
f(T_f^-)+g(t-T_f), & T_f \leq t < T_f+T_g
\end{cases}.$$
\item \textit{Time marked continuation}: Define $\ha{\oplus}: \Sigma^{\oplus} \times \Sigma_{\oplus} \to \Sigma \times [0,\infty)$ by $f \ha{\oplus} g=(f \oplus g, T_f)$. 
\end{itemize}

We will also have to work with random measures, which are called kernels. More precisely, suppose $(U, \cal{U})$ and $(V,\cal{V})$ are measurable spaces. A kernel from $(U,\cal{U})$ to $(V,\cal{V})$ is a map $\nu:(U,\cal{V}) \to [0,\infty)$ such for each $u \in U$, $\nu(u, \cdot): \mathcal{V} \to [0,\infty)$ is a measure and for each $E \in \mathcal{V}$, the function $\nu(\cdot, E): U \to [0,\infty)$ is $\mathcal{U}$-measurable. If $\mu$ is a measure on $(U,\mathcal{U})$, then $\nu$ is called a $\mu$-kernel if it is a kernel on the $\mu$-completion of $(U,\mathcal{U})$. We say that $\nu$ is a finite $\mu$-kernel if $\nu(u,V)<\infty$ for $\mu$-a.s. $u \in \mathcal{U}$, and we say $\nu$ is a $\sigma$-finite $\mu$-kernel if $V=\bigcup_{n=1}^{\infty}F_n$ such that for each $n$, for $\mu$-a.s. $u \in U$, we have $\nu(u,F_n) < \infty$.

Combining kernels with measures, we have the following operations for measures:
\begin{itemize}
\item If $\mu$ is a $\sigma$-finite measure on $(U,\mathcal{U})$ and $\nu$ is a $\sigma$-finite $\mu$-kernel from $(U,\mathcal{U})$ to $(V,\mathcal{V})$, then we define $\mu \otimes \nu$ on $U \times V$ by
$$\mu \otimes \nu(E \times F) = \int_E \nu(u,F)d\mu(u).$$ In this case, $\mu\cdot\nu:=\int \nu(X,\cdot)\mu(dX)$ is the marginal measure of $\mu\otimes \nu$ on $V$.
\item If $\nu$ is a $\sigma$-finite measure on $V$ and $\mu$ is a $\sigma$-finite $\nu$-kernel from $(V,\mathcal{V})$ to $(U,\mathcal{U})$, we define $\mu \overleftarrow\otimes \nu$ on $U \times V$ to be the pushforward measure of $\nu\otimes \mu$ under the map $(v,u)\mapsto (u,v)$.
\end{itemize}

\begin{remark}
  One may sample $(X,Y)$ according to the ``law'' $\mu\otimes\nu$ by two steps. First, sample $X$ according to the ``law'' $\mu$. Second, given $X$, sample $Y$ according to the ``law'' $\nu(X,\cdot)$. We use the quotation marks because the $\mu$ and $\nu(X,\cdot)$ are in general not probability measures. One should be careful that the marginal ``law'' of $X$ is changed after the second step unless  for $\mu$-a.s. $X$, $\nu(X,\cdot)$ is a probability measure.
\label{sample}
\end{remark}
 
Combining the operations on $\Sigma$ and the language of kernels, we have the following operations for measures.

\begin{itemize}
\item If $\nu$ is a $\sigma$-finite $\mu$-kernel from $\Sigma$ to $(0,\infty)$, define a measure $\mathcal{K}_{\nu}(\mu)$ on $\Sigma$ to be the pushforward measure of $\mu \otimes \nu$ under the map $\mathcal{K}:\Sigma \times (0,\infty) \to \Sigma$ given by $(f, r) \to \mathcal{K}_r(f)$.
\item If $\mu$ is a $\sigma$-finite measure supported by $\Sigma^\oplus$, and $\nu$ is a $\sigma$-finite $\mu$-kernel from $\Sigma^\oplus$ to $\Sigma_\oplus$, then $\mu \oplus \nu$ and $\mu\ha{\oplus}\nu$ are defined by the pushforward measures $\oplus_*(\mu \otimes \nu)$ and $\ha\oplus_*(\mu\otimes\nu)$, respectively. 
\end{itemize}

The following propositions are \cite[Propositions 2.1 and 2.2]{Zhan decomp}.
\begin{proposition}\label{prop: decomp prop 2.1}
Let $\mu$ be a measure on $(\Sigma, \mathcal{F})$ which is $\sigma$-finite on $\mathcal{F}_0$. Let $(\Upsilon, \mathcal{G})$ be a measurable space. Let $\nu: \Upsilon \times \mathcal{F} \to [0,\infty]$ be such that for every $v \in \Upsilon$, $\nu(v, \cdot)$ is a finite measure on $\mathcal{F}$ that is locally absolutely continuous  with respect to $\mu$. Moreover, suppose that the local Radon-Nikodym derivatives are equal to $(M_t(v, \cdot))_{t>0}$, where $M_t:\Upsilon \times \Sigma \to [0,\infty)$ is $\mathcal{G}\times \mathcal{F}_t$ measurable for every $t \geq 0$. The $\nu$ is a kernel from $(\Upsilon, \mathcal{G})$ to $(\Sigma,\mathcal{F})$. Moreover, if $\xi$ is a $\sigma$-finite measure on $(\Upsilon, \mathcal{G})$ such that $\mu$-a.s., $\int_{\Upsilon}M_t(v, \cdot)d\xi(v)<\infty$, then $\xi \cdot \nu \lhd \mu$, and the local Radon-Nikodym derivatives are $\int_{\Upsilon} M_t(v,\cdot)d\xi(v)$ for $0 \leq t <\infty$.
\end{proposition}

\begin{proposition}\label{prop: decomp prop 2.2}
Let $\mu$ be a probability measure on $(\Sigma,\mathcal{F})$. Let $\xi$ be a $\mu$-kernel from $(\Sigma, \mathcal{F})$ to $(0,\infty)$ that satisfies $\E_{\mu}[\xi((0,\infty))]<\infty$. Then $\mathcal{K}_{\xi}(\mu) \lhd \mu$, and the local Radon-Nikodym derivatives are $\E_{\mu}[\xi((t,\infty))|\mathcal{F}_t]$ for $0 \leq t <\infty$.
\end{proposition}

For $\kappa>0$, we use $\P^\kappa_{B}$ to denote the law of $(\sqrt{\kappa}B_t)_{t >0}$, and use $\E^\kappa_B$ to denote the corresponding expectation, where $(B_t)_{t>0}$ is a standard one dimensional Brownian motion. We will omit $\kappa$ when it is fixed in the context. Then $\P_{B}$ is supported on $C[0,\infty) \cap \Sigma_{\oplus}$. 

The following propositions are \cite[Propositions 2.3 and 2.4]{Zhan decomp}. The first extends Girsanov theorem to a statement about local absolute continuity, and the second extends the strong Markov property of Brownian motion.

\begin{proposition}\label{prop: decomp 2.3}
Let $\P^\kappa_\sigma$ denote the law of the process $(X_t)_{0 \leq t <T}$, which satisfies the SDE:
	$$dX_t=\sqrt{\kappa}dB_t + \sigma_t dt, \quad 0 \leq t <T, \quad X_0=0,$$
	where $T$ is a stopping time.
Suppose that $(M_t)_{0\le t<T}$ is a positive local martingale that satisfies the SDE
$$dM_t=M_t \frac{\sigma_t}{\sqrt{\kappa}}dB_t, \quad 0 \leq t < T.$$ 
 Then $\P^{\kappa}_\sigma \lhd \P^\kappa_{B}$, and the local Radon-Nikodym derivatives are
$$\frac{d\P^{\kappa}_\sigma|_{\mathcal{F}_t \cap \Sigma_t}}{d\P^\kappa_{B}|_{\mathcal{F}_t \cap \Sigma_t}}= \I_{\{T>t\}}\frac{M_t}{M_0}, \quad 0 \leq t < \infty.$$
\end{proposition}


\begin{proposition}\label{prop: decomp prop 2.4}
Let $(\theta_t)_{0 \leq t < \infty}$ be a right-continuous increasing adapted process defined on $\Sigma$ that satisfies $\theta_0=\theta_{0^+}=0$ and $\E^\kappa_{B}[\theta_{\infty}]<\infty$. Let $d_\theta$ be the kernel such that $d_\theta(f,\cdot)$ is the measure induced by the monotonic function $t\mapsto\theta_t(f)$ on $[0,\infty)$. Then
$$\mathcal{K}_{d\theta}(\P^\kappa_{B})\ha{\oplus}\P^\kappa_{B}=\P^\kappa_{B} \otimes d\theta.$$
Thus, $\mathcal{K}_{d\theta} \oplus \P^\kappa_{B} \ll \P^\kappa_{B}$, and $\theta_{\infty}$ is the Radon-Nikodym derivative.
\end{proposition}

\subsection{Laws of driving functions}
Let $\kappa>0$ and $\rho_+,\rho_-\in\R$. For $x>0>y$,
let $\P^{\rho_+,\rho_-}_{x,y}$ be the law of the driving function for the BSLE$_{\kappa}(\rho_+,\rho_-)$ process starting from $(0;x,y)$. We have discussed before Definition \ref{extended} that $\P^{\rho_+,\rho_-}_{x,y}$ is supported by $\Sigma^\oplus$. So $\P^{\rho_+,\rho_-}_{x,y}\oplus \P_{B}$ is well defined, which is the law of an extended BSLE$_{\kappa}(\rho_+,\rho_-)$ process.


Fix $\rho_+,\rho_-\in\R$. Observe that $(x,y)\mapsto\P^{\rho_+,\rho_-}_{x,y}\oplus \P_B$ is a probability kernel from the fourth quadrant $Q_4:=\R_+\times\R_-$ to $\Sigma$. Recall the $(\rho_+,\rho_-)$-BSLE$_\kappa$ Green's function $G^{\rho_+,\rho_-}(x,y)$ defined after Proposition \ref{prop: general Greens functions}. 
Let $A$ denote the Lebesgue measure on $\R^2$.

\begin{definition}
We say that BSLE$_{\kappa}$ admits a BSLE$_{\kappa}(\rho_+,\rho_-)$ decomposition if there is a $\P_B$-kernel $\nu^{\rho_+,\rho_-}$ from $C[0,\infty)$ to $Q_4$ such that
\BGE \P_B(d\lambda) \otimes \nu^{\rho_+,\rho_-}(\lambda,d(x,y))=(\P^{\rho_+,\rho_-}_{x,y}\oplus \P_B)(d\lambda)\overleftarrow{\otimes} G^{\rho_+,\rho_-}(x,y)\cdot \I_{Q_4} A(d(x,y)).\label{decomposition}\EDE
\end{definition}

In the spirit of Remark \ref{sample}, we may interpret (\ref{decomposition}) as follows. It means that we have two methods to  sample the same measure on the space $C([0,\infty))\times Q_4$. One is first sample $\lambda\in C([0,\infty))$ according to the law of $(\sqrt\kappa B_t)$, and given $\lambda$ then sample $(x,y)\in Q_4$ according to the law $\nu^{\rho_+,\rho_-}(\lambda,\cdot)$. The other is first sample $(x,y)\in Q_4$ according to the measure $G^{\rho_+,\rho_-}\cdot \I_{Q_4} A$, i.e., the Lebesgue measure restricted to $Q_4$ weighted by the Green's function, and given $(x,y)$ then sample $\lambda$ according to the law of the extended BSLE$_\kappa({\rho_+,\rho_-})$ driving function started from $(0;x,y)$.


This is a backward analogue to the definition in \cite{Zhan decomp}, which defines what it means for SLE$_{\kappa}$ to admit an SLE$_{\kappa}(\rho)$ decomposition. In \cite{Zhan decomp}, it is proven that SLE$_{\kappa}$ admits both an SLE$_{\kappa}(\kappa-8)$ decomposition and an SLE$_{\kappa}(-8)$ decomposition. In the former case, $\nu$ corresponds to the natural parametrization, and in the latter $\nu$ corresponds to a constant times the capacity parametrization.

Later, we will prove that (\ref{decomposition}) holds for $\kappa\le 4$ and $\rho_+=\rho_-=-4$, and the kernel $\nu$ is given by the pushforward measure $C\Phi^\lambda_*(\mA|_{\R_+})$, where $C$ is a constant depending on $\kappa$, $\mA$ is the Lebesgue measure on $\R$, and $\Phi^\lambda$ is the welding curve generated by $\lambda$ as  in Definition \ref{welding}. 

We believe that the decomposition also holds for some other parameters $(\rho_+,\rho_-)$, probably $\rho_+=\rho_-=-\frac{\kappa}{2}-4$, and the corresponding kernel $\nu$ is related to the ``natural parametrization" of BSLE$_\kappa$.

\section{Main Theorem}
The goal of this section is to show that, for $\kappa\in(0,4]$, BSLE$_{\kappa}$ admits a BSLE${_{\kappa}}(-4,-4)$ decomposition, and to show the relationship between this decomposition and the capacity time for the welding curve. For this section, we fix $\kappa\in(0,4]$ and write $G$, $M_t$ and $\P_{x,y}$ for $G^{-4,-4}$, $M^{-4,-4}_t$ and $\P^{-4,-4}_{x,y}$, respectively. From Propositions  \ref{prop: general Greens functions} and \ref{prop: decomp 2.3}, we know that $\P_{x,y}\lhd \P_{B}$, and the local Radon-Nikodym derivative is $\I_{\tau_x\wedge \tau_y>t}\frac{M_t(x,y)}{G(x,y)}$, where $G(x,y)=|x|^{\frac{4}{\kappa}}|y|^{\frac{4}{\kappa}}|x-y|^{-\frac{8}{\kappa}}$.
As before, $T$ is the lifetime of the BSLE${_{\kappa}}(-4,-4)$ process. The following is the main technical lemma of the paper.

\begin{lemma}\label{lem: int P is finite}
$\int_{\R_-}\! \int_{\R_+}\P_{x,y}[T\le 1] dxdy <\infty$.
\end{lemma}
\begin{proof}
%
Using (\ref{eq: tau eq 2}), we get $\P_{rs,r(s-1)}[T\le 1]=\P_{s,s-1}[T\le r^{-2}]$.
Performing the change of coordinates $x=rs, y=r(s-1)$ and writing $\P_s$ for $\P_{s,s-1}$ and $\E_s$ for $\E_{s,s-1}$, we get that
$$\int_{\R_-}\int_{\R_+} \P_{x,y}[T\le t] dxdy  = \int_0^1\! \int_0^{\infty} \P_{s}[T\le r^{-2}]r dr ds.$$ For any fixed $s\in(0,1)$, observe that
$$\int_0^{\infty}\P_s[T \leq r^{-2}]r dr=\int_0^{\infty} \P_s\[T^{-1} \geq r^2\]rdr=\frac{1}{2} \E_s\[ T^{-1} \].$$  Thus, it suffices to show that $\int_0^1 \E_s\[ T^{-1} \]ds<\infty$.

We use the diffusion process $(\ha W_t)$. Since $a=b=-4$, $x=s$ and $y=s-1$,  equations (\ref{eq: radial Bessel}) and (\ref{eq: tau eq 2})  reduce to
\begin{equation}\label{eq: decomp capacity diffusion}
d\ha{W}_t = -\sqrt{1-\ha{W}_t^2}d\ha B_t - \frac{4}{\kappa}\ha{W}_t dt;
\end{equation}
\BGE \label{eq: T-reduced}
T=\frac{1}{4\kappa}\int_0^\infty e^{-\frac 4\kappa t}(1-\ha W_t^2)dt.
\EDE
Moreover, the law $\P_s$ corresponds to the initial value $\ha W_0=2s-1$. So we will use $\ha W^s_t$ to denote a process that satisfies (\ref{eq: decomp capacity diffusion}) with initial value $2s-1$ at time $0$.

Fix $\delta\in(0,1/2)$ is small. First suppose $s \in [\delta, 1-\delta]$.
We may couple  two processes $(\ha W^{s}_t)$ and  $(\ha W^{\delta}_t)$ such that $|\ha W^{s}_t|\le |\ha W^{\delta}_t|$ for all $t\in[0,\infty)$. For the construction of the coupling, we start the two processes independently, and after the first time that $|\ha W^{s}_t|= |\ha W^{\delta}_t|$, we continue the two processes such that the ratio between them is constant $1$ or $-1$. This is possible since the SDE (\ref{eq: decomp capacity diffusion}) is symmetric about $0$. Using (\ref{eq: T-reduced}) we see that
\BGE \E_{s}[T^{-1}]\le \E_{\delta}[T^{-1}],\quad s\in[\delta,1-\delta].\label{s-big}\EDE

Now suppose $s\in(0,\delta)$. Define $\tau_{\delta}=\inf\{ t>0: \ha{W}^s_t=2\delta-1 \}$. From (\ref{eq: T-reduced}), we get
$$T =\frac{1}{4\kappa} \int_0^{\infty} e^{-\frac{4}{\kappa}t} (1-(\ha{W}^s_t)^2)dt \geq \frac{1}{4\kappa}\int_{\tau_{\delta}}^{\infty} e^{-\frac{4}{\kappa} t} (1-(\ha{W}^s_t)^2)dt=e^{-\frac{4}{\kappa}\tau_{\delta}} \frac{1}{4\kappa} \int_0^{\infty} e^{-\frac{4}{\kappa}t}  (1-(\ha{W}^s_{\tau_{\delta}+t})^2)dt.$$
By the Markov property of $(\ha W_t)$, we know that $(\ha{W}^s_{\tau_{\delta}+t})$ has the law of $(\ha W^\delta_t)$. Thus,
\BGE \E_s[T^{-1}]\le \E_s[e^{\frac{4}{\kappa}\tau_{\delta}}] \cdot \E_\delta[T^{-1}],\quad s\in(0,\delta).\label{s-small}\EDE

We will show that there is some $\delta_0\in(0,1/2)$ such that $\E_{\delta_0}[T^{-1}]<\infty$ and $\E_s[e^{\frac{4}{\kappa}\tau_{\delta_0}}]$, $0<s<\delta_0$, are uniformly bounded by a finite number. Once this is done, using (\ref{s-big},\ref{s-small}) and that $\E_{1-s}[T^{-1}]=\E_s[T^{-1}]$ we then get the finiteness of $\int_0^1 \E_s\[ T^{-1} \]ds$.

First, we bound $\E_s[e^{\frac{4}{\kappa}\tau_{\delta_0}}]$, $0<s<\delta_0$. Since the process $(\ha W_t)$ equals to the cosine of a radial Bessel process of dimension $8/\kappa\ge 2$ (\cite[Formula (8.2)]{ZhanE}), it makes sense to define the process $(\ha W^0_t)$. Such process starts from $-1$ at time $0$, and stays in $(-1,1)$ and satisfies (\ref{eq: decomp capacity diffusion}) after the initial time.
One may couple $(\ha W^s_t)$ with $(\ha W^0_t)$ such that $\ha W^0_t\le \ha W^s_t$ for all $t\ge 0$, and so the $\tau_\delta$ for $(\ha W^0_t)$  is bigger than that for $(\ha W^s_t)$. So it suffices to bound $\E_0[e^{\frac{4}{\kappa}\tau_{\delta}}]$ for some $\delta\in(0,1/2)$.
We now estimate $\P_0[\tau_\delta>n]$ for $n\in\N$. By the Markov property of $(\ha W_t)$ and the fact that the $\tau_\delta$ for $(\ha W^s_t)$ is stochastically bounded by the $\tau_\delta$ for $(\ha W^0_t)$, we get $\P_0[\tau_\delta>n]\le \P_0[\tau_\delta>1]^n$. Therefore,
\BGE \E_0[e^{\frac{4}{\kappa}\tau_{\delta}}]\le \sum_{n=1}^\infty e^{\frac{4}{\kappa}n} \P_0[\tau_\delta>n-1]
\le \sum_{n=1}^\infty e^{\frac{4}{\kappa}n}\P_0[\tau_\delta>1]^{n-1}.\label{exp}\EDE
We have $\P_0[\tau_\delta>1]\le \P_0[\sup_{0\le t\le 1}\ha W^0_t<2\delta-1]\to 0$ as $\delta\to 0^+$ since $\ha W_t>-1$ for $t>0$. So there is $\delta_0\in(0,1/2)$ such that $\P_0[\tau_{\delta_0}>1]<e^{-\frac 4\kappa}$. From (\ref{exp}) we get $\E_0[e^{\frac{4}{\kappa}\tau_{\delta_0}}]<\infty$, and so $\E_s[e^{\frac{4}{\kappa}\tau_{\delta_0}}]$, $0<s<\delta_0$, are uniformly bounded.

It remains to show that $\E_{\delta_0}[T^{-1}]<\infty$. Fix $\delta'\in(0,\delta_0)$. Let $\ha\tau_{\delta'}=\inf\{t>0: \ha W^{\delta_0}_t\in\{2\delta'-1,1-2\delta'\}\}$. Then $|W^{\delta_0}_t|\le 1-2\delta'$ on $[0,\ha\tau_{\delta'}]$.  From (\ref{eq: T-reduced}), we get
$$T\ge \frac 1{4\kappa} \int_0^{\ha\tau_{\delta'}}e^{-\frac 4\kappa t}(1-(1-2\delta')^2)dt\ge \frac{\delta'(1-\delta')}{\kappa e^{4/\kappa}} \min\{1,\ha\tau_{\delta'}\}.$$
Thus, $\E_{\delta_0}[T^{-1}]\le \frac{\kappa e^{4/\kappa}}{\delta'(1-\delta')}\E_{\delta_0}[\ha\tau_{\delta'}^{-1}+1]$. So it suffices to show that $\E_{\delta_0}[\ha\tau_{\delta'}^{-1}]<\infty$.

For that purpose, consider the process $\ha V^{\delta_0}_t=\arccos \ha W^{\delta_0}_t$, which satisfies the SDE
\BGE d\ha V^{\delta_0}_t=d\ha B_t +\frac{8/\kappa-1}2\cot(\ha V^{\delta_0}_t)dt\label{V}\EDE
with initial value $\ha V^{\delta_0}_0=\arccos(2\delta_0-1)$. Now $\ha\tau_{\delta'}$ is the first time that $\ha V^{\delta_0}_t$ reaches $\arccos(2\delta'-1)$ or $\pi-\arccos(2\delta'-1)$. Since $|\cot(\ha V^{\delta_0}_t)|$ on $[0,\ha\tau_{\delta'}]$ is bounded by $\cot(\arccos(1-2\delta'))$, from (\ref{V})  we have
$$\ha B_t-At \leq \ha V^{\delta_0}_t-\ha V^{\delta_0}_0 \leq \ha B_t + At,\quad 0\le t\le \ha\tau_{\delta'},$$ where $A:=\frac{8/\kappa-1}2 \cot(\arccos(1-2\delta')) \in(0,\infty)$. Setting $X=|\arccos(2\delta_0-1)-\arccos(2\delta'-1)|>0$, we get
$$|\ha B_{\ha\tau_{\delta'}}|\ge | \ha V^{\delta_0}_{\ha\tau_{\delta'}}-\ha V^{\delta_0}_0 |-A {\ha\tau_{\delta'}}\ge X-A {\ha\tau_{\delta'}}.$$
Thus, if ${\ha\tau_{\delta'}}< \eps<\frac{X}{2A}$, then $\sup_{0\le t\le \eps}|\ha B_t|\ge X-A\eps>X/2$. By scaling property, symmetric property and reflection property of Brownian motion, we get
$$\P_{\delta_0}[{\ha\tau_{\delta'}}< \eps]\le \P[\sup_{0\le t\le 1} |\ha B_t|>\frac{X}{2\sqrt\eps}]\le 2 \P[\sup_{0\le t\le 1} \ha B_t>\frac{X}{2\sqrt\eps}]\le 4\P[ \ha B_1>\frac{X}{2\sqrt\eps}],\quad 0<\eps<\frac X{2A}.$$
As $\eps\to 0^+$, the RHS of the above formula decays like a constant times $e^{-\frac{X^2}{8\eps}}$. So we get  the finiteness of $\E_{\delta_0}[\ha\tau_{\delta'}^{-1}]$. The proof is now complete.
\end{proof}

\begin{corollary} \label{cor-Ct}
  Let $C_t=\int_{\R_-}\!\int_{\R_+} G(x,y) \P_{x,y}[T \leq t]dxdy$.
Then $C_t=tC_1$, for all $t\ge 0$, and $C_1\in(0,\infty)$.
\end{corollary}
\begin{proof}   That $C_t=tC_1$ follows from (\ref{eq: tau eq 2}) and the fact that $G(x,y)$ is homogeneous of degree $0$. That $C_1>0$ follows from the fact that $T\le (x-y)^2/16\le 1$ under $\P_{x,y}$ when $x-y\le 4$ (see the discussion before Definition \ref{extended}) and that $G>0$. That $C_1<\infty$ follows from Lemma \ref{lem: int P is finite} and that $G< 1$.
\end{proof}

We now make use of the results in Section \ref{subsection: random lifetime}. For any $N>0$ and $x>0>y$, define the measure $P^N_{x,y}$ by $P^N_{x,y}[E]=\P_{x,y}[E \backslash \Sigma_N]$, where $\Sigma_N = \{f \in \Sigma : T_f >N\}$. Note that this is not a probability measure.
Let $E \in \mathcal{F}_t \cap \Sigma_t$. If $t \geq N$, then $P^N_{x,y}[E]=0$. If $t < N$, then
$$P^N_{x,y}[E]:=\P_{x,y}[E]-\P_{x,y} [E \cap \Sigma_N]=\int_E\I_{\tau_x\wedge \tau_y>t} \frac{M_t(x,y)}{G(x,y)} d\P_{B}-\int_{E \cap \Sigma_N}\I_{\tau_x\wedge \tau_y>N} \frac{M_N(x,y)}{G(x,y)}d\P_{B}$$
$$= \int_E \frac{ \I_{\tau_x\wedge \tau_y>t}}{G(x,y)} ( M_t(x,y)-\E_{B}[\I_{\tau_x\wedge \tau_y>N}  M_N(x,y)|\mathcal{F}_t])d\P_{B}.$$
Therefore, we conclude that
\begin{equation}\label{eq: capacity local rn derivative}
\frac{d P^N_{x,y}|_{ \mathcal{F}_t \cap \Sigma_t}}{d\P_{B}|_{ \mathcal{F}_t \cap \Sigma_t}} =  \frac{\I_{\{t <N\wedge \tau_x\wedge \tau_y\}}}{G(x,y)} \( M_t(x,y)-\E_{B}[\I_{\tau_x\wedge \tau_y>N} M_N(x,y)| \mathcal{F}_t] \).
\end{equation}

It is straightforward to check that, for a fixed $t_0\ge 0$, the maps $(f^{t_0}_t(z):= f_{t_0+t}\circ f_{t_0}^{-1}(z+\lambda_{t_0})-\lambda_{t_0})_{t\ge 0}$ are the backward Loewner maps driven by $(\lambda^{t_0}_t:=\lambda_{t_0+t}-\lambda_{t_0})_{t\ge 0}$. Since $\lambda_t=\sqrt\kappa B_t$, from the Markov property of Brownian motion, we find that $(\lambda^{t_0}_t)$ has the same law as $(\lambda_t)$, and is independent of $\lambda_s$, $0\le s\le t_0$. With $X_{t_0}=f_{t_0}(x)-\lambda_{t_0}$ and $Y_{t_0}=f_{t_0}(y)-\lambda_{t_0}$, we may rewrite $M_{t_0+t}(x,y)$ as
$$M_{t_0+t}(x,y)
=f_{t_0}'(x)f_{t_0}'(y) M^{\lambda^{t_0}_\cdot}_t(X_{t_0},Y_{t_0}),$$
where $M^{\lambda^{t_0}_\cdot}_t(\cdot,\cdot)$ is the $M_t(\cdot,\cdot)$ function generated by $\lambda^{t_0}_\cdot$. Define $G_t$ on $Q_4$ by
$$ {G_t(x,y)}={G(x,y)}\P_{x,y}[T \leq t]={G(x,y)}\P_{x,y}[\Sigma\sem \Sigma_t]= {G(x,y)-\E_{B}[\I_{\tau_x\wedge \tau_y>t} M_t(x,y)]}.$$
Then we have
\BGE C_t=\int_{\R_-}\!\int_{\R_+} G_t(x,y) dx dy\label{CtGt}\EDE
\BGE M_t(x,y)-\E_{B}[\I_{\tau_x\wedge \tau_y>N} M_N(x,y)| \mathcal{F}_t] =f_t'(x)f_t'(y) G_{N-t}(f_t(x)-\lambda_t,f_t(y)-\lambda_t).\label{Mt-condition}\EDE


Let $P^N=\int_{\R_-}\!\int_{\R_+} G(x,y) P^N_{x,y}dxdy$.
From (\ref{eq: capacity local rn derivative},\ref{CtGt},\ref{Mt-condition}), Proposition \ref{prop: decomp prop 2.1} and Corollary \ref{cor-Ct}, we know that $P^N\lhd \P_B$, and $$ \frac{d P^N |_{ \mathcal{F}_t \cap \Sigma_t}}{d\P_{B}|_{ \mathcal{F}_t \cap \Sigma_t}} =\int_{-\infty}^{c_t}\!\int_{d_t}^\infty \I_{t<N} (M_t(x,y)-\E_{B}[M_N(x,y)| \mathcal{F}_t])dxdy $$
$$=\I_{t<N} \int_{-\infty}^{c_t}\!\int_{d_t}^\infty f_t'(x)f_t'(y) G_{N-t}(f_t(x)-\lambda_t,f_t(y)-\lambda_t)dxdy $$
$$=\I_{t<N} \int_{-\infty}^{0}\!\int_{0}^\infty G_{N-t}(X,Y)dXdY=\I_{t<N} C_{N-t}=((N-t)\vee 0)C_1,$$
where the first equality of the last line follows from a change of variables and the fact that $f_t-\lambda_t$ maps $(-\infty,c_t)$ and $(d_t,\infty)$ onto $\R_-$ and $\R_+$, respectively.

Let $\mA$ denote the Lebegue measure on $\R$. From Proposition \ref{prop: decomp prop 2.2} we get $P^N={\cal K}_{C_1\mA|_{[0,N]}}(\P_B)$. Using Proposition \ref{prop: decomp prop 2.4}, we then get $P^N\ha\oplus \P_B=\P_B\otimes C_1\mA|_{[0,N]}$. Using the definition of $P^N$, we get
$$\int_{\R_-}\!\int_{\R_+} G(x,y) P^N_{x,y}\ha \oplus\P_B dxdy =\P_B\otimes C_1\mA|_{[0,N]}.$$
Applying the map $(\lambda,t)\mapsto (\lambda,\Phi^\lambda(t))$ to both sides of the above formula, where $\Phi^\lambda$ is the welding curve generated by $\lambda$ as in Definition \ref{welding}, and using Remark \ref{<-kappa/2-2}, we get
$$(P^N_{x,y}\ha \oplus\P_B)(d\lambda)\overleftarrow{\otimes}{\bf 1}_{Q_4} G(x,y)\cdot A(dxdy)=
C_1\P_B(d\lambda)\otimes \Phi^\lambda_*(\mA|_{[0,N]})(dxdy).$$
Since $P^N_{x,y}=\P_{x,y}|_{\Sigma\sem \Sigma_N}$, by sending $N\to\infty$, we get the following theorem, which tells us that BSLE$_\kappa$ admits BSLE$_\kappa(-4,-4)$ decomposition, and the corresponding kernel is associated with the capacity parametrization. It is similar to \cite[Theorem 5.1]{Zhan decomp}.

\begin{Theorem}\label{thm: Decomp BSLE capacity} Let $\kappa\in(0,4]$. Let $\P_B$ denote the law of $(\sqrt\kappa B_t)_{t\ge 0}$, where $(B_t)$ is a standard Brownian motion. For $x>0>y$, let $\ha P_{x,y}$ denote the law of the driving function for the extended BSLE$_\kappa(-4,-4)$ process started from $(0;x,y)$. Let $\mA$ and $A$ denote the Lebegue measure on $\R$ and $\R^2$, respectively. Let $G(x,y)=|x|^{\frac{4}{\kappa}}|y|^{\frac{4}{\kappa}}|x-y|^{-\frac{8}{\kappa}}$ for $(x,y)\in \R_+\times\R_-$. Let $\Phi^\lambda$ be the welding curve from $\R_+$ into $\R_+\times\R_-$ generated by $\lambda$ as in Definition \ref{welding}. Then there is a constant $C_1\in(0,\infty)$ depending only on $\kappa$ such that
\begin{equation}\label{eq: decomp BSLE capacity}
\ha\P_{x,y}(d\lambda)\overleftarrow{\otimes}{\bf 1}_{ \R_+\times\R_-} G(x,y)\cdot A(dxdy)=C_1
\P_B(d\lambda)\otimes \Phi^\lambda_*(\mA|_{\R_+})(dxdy).
\end{equation}
\end{Theorem}

By restricting the measures on both sides of (\ref{eq: decomp BSLE capacity}) to $\Sigma\times U$ for any given measurable set $U\subset \R_+\times\R_-$, and looking at the marginal measures of the second coordinate  and the first coordinate, respectively, we get the following corollary.

\begin{corollary}
Let $\kappa\in(0,4]$. There is a constant $C_1\in(0,\infty)$ depending on $\kappa$ such that the following holds.
\begin{enumerate}
  \item [(i)]  For any measurable set $U\subset \R_+\times\R_-$, we have the expectation of the total (capacity) time that a BSLE$_\kappa$ welding curve spends in $U$:
  $$\E_B[\mA(\{t\ge 0:\Phi(t)\in U\})]=C_1^{-1}\int\!\int_U G(x,y)dxdy,$$ where $G(x,y)=|x|^{\frac{4}{\kappa}}|y|^{\frac{4}{\kappa}}|x-y|^{-\frac{8}{\kappa}}$.
  \item [(ii)] If, in addition, $\int\!\int_U G(x,y)dxdy<\infty$, then the integral of the laws of the  extended BSLE$_{\kappa}(-4,-4)$ process started from $(0;x,y)$    against the measure $\I_U G (x,y)dxdy$ is a bounded measure, which is absolutely continuous with respect to the law of the BSLE$_{\kappa}$ process, and the Radon-Nikodym derivative is $C_1 \mA(\{t\ge 0:\Phi(t)\in U\})$.
\end{enumerate}
\end{corollary}

\end{document}